\documentclass[11pt,reqno]{amsart}
\usepackage[utf8]{inputenc}
\usepackage{a4wide}
\usepackage{appendix}

\usepackage{graphicx}
\usepackage{enumerate}
\usepackage{hyperref}
\usepackage{cleveref}
\usepackage{amsmath,amsopn,amssymb,amsfonts,stmaryrd}
\usepackage{verbatim}
\usepackage{amsthm}
\usepackage{mathtools}
\usepackage{color}
\usepackage{enumitem}
\usepackage[framemethod=TikZ]{mdframed}
\usepackage{bbm}
\usepackage{mathrsfs}
\usepackage{booktabs}
\usepackage{caption}
\usepackage{bm}
\usepackage{tensor}
\usepackage{float}

\makeatletter
\@namedef{subjclassname@2020}{%
	\textup{2020} Mathematics Subject Classification}
\makeatother

\newcommand\restr[2]{\ensuremath{\left.#1\right|_{#2}}}

\theoremstyle{plain}
\newtheorem{thm}{Theorem}[section]
\newtheorem{cor}[thm]{Corollary}
\newtheorem{lem}[thm]{Lemma}
\newtheorem{prop}[thm]{Proposition}

\newtheorem{rem}[thm]{Remark}

\theoremstyle{definition}

\newtheorem{defn}[thm]{Definition}

\numberwithin{equation}{section}

\def\d{\delta}

\def\d{\delta}

\def\RR{{\mathbb R}}

\def\ZZ{{\mathbb Z}}
\def\d{{\mathrm{d}}}

\def\id{\mathrm{id}}

\def\Sd{{\mathbb{S}^3}}

\def\SS{{\mathbb{S}}}

\def\Diff{\mathrm{Diff}}

\newcommand{\gc}{\gamma}
\newcommand{\dgc}{\dot{\gc}}

\setlist[itemize]{noitemsep, topsep=0pt}

\makeatletter
\newcommand{\vast}{\bBigg@{2}}
\newcommand{\Vast}{\bBigg@{5}}
\makeatother

\newcommand{\RNum}[1]{\uppercase\expandafter{\romannumeral #1\relax}}
\allowdisplaybreaks

\title[The magnetic Euler--Arnold equation]{On geometric hydrodynamics and \\infinite--dimensional magnetic systems} 

\author{L. Maier}
\address{Faculty of Mathematics and Computer Science,
	University of Heidelberg,
	Im Neuenheimer Field 205,
	69120 Heidelberg, Germany}
\email{lmaier@mathi.uni-heidelberg.de}


\subjclass[2020]{}

\begin{document}
\maketitle

	\renewcommand{\abstractname}{Abstract}
	\begin{abstract}
	In this article, we combine V. Arnold's celebrated approach via the Euler–Arnold equation—describing the geodesic flow on a Lie group equipped with a right-invariant metric \cite{Arnold66}—with his formulation of the motion of a charged particle in a magnetic field \cite{ar61}. We introduce the \emph{magnetic Euler–Arnold equation}, which is the Eulerian form of the magnetic geodesic flow for an infinite-dimensional magnetic system on a Lie group endowed with a right-invariant metric and a right-invariant closed two-form serving as the magnetic field.

As an illustration, we demonstrate that the Korteweg–de Vries equation, the generalized Camassa–Holm equation, the infinite conductivity equation, and the global quasi-geostrophic equations can all be interpreted as magnetic Euler–Arnold equations. In particular, we obtain both local and global well-posedness results for the magnetic Euler–Arnold equation associated with the global quasi-geostrophic equations.
	\end{abstract}
	\maketitle
	\renewcommand{\abstractname}{Abstract}
    \section{Introduction}
Since V.~Arnold’s seminal discovery~\cite{Arnold66}—that the Euler equations of hydrodynamics, which govern the motion of an incompressible and inviscid fluid in a fixed domain (with or without boundary), can be interpreted as the geodesic equations on the group of volume-preserving diffeomorphisms of the domain, endowed with a right-invariant Riemannian metric (specifically, the \( L^2 \)-metric)—many partial differential equations (PDEs) arising in mathematical physics have been reinterpreted within a similar geometric framework. These equations are formulated as geodesic equations on infinite-dimensional Lie groups equipped with a right-invariant Riemannian metric; see, for example,~\cite{AK98, Khesin-Mis-Mod-inf-Newton, Vi08} and the references therein.

In~\cite{Khesin-Mis-Mod-inf-Newton}, it is further shown that many PDEs in mathematical physics can be formulated as infinite-dimensional Newton equations. From a physical perspective, this provides a natural extension of the geodesic framework: while the geodesic equation describes the motion of a free particle, Newton’s equation captures the dynamics of a particle under the influence of a potential force.

From this viewpoint, a natural next step is to study the motion of a charged particle in a magnetic field. Mathematically, this problem is formulated within Hamiltonian dynamics, specifically through the theory of magnetic systems—pioneered by V.~Arnold in~\cite{ar61}. The corresponding equations of motion, known as the \emph{magnetic geodesic equations}, may be interpreted as geodesic equations modified by the \emph{Lorentz force}, induced by the presence of an external magnetic field.

In~\cite{M24}, the author constructed the first example of a PDE admitting a formulation as an infinite-dimensional magnetic geodesic equation: the so-called magnetic two-component Hunter–Saxton system. In the present paper, we show that this example fits into a broader and more general framework. By combining the ideas of V.~Arnold~\cite{ar61,Arnold66}, we introduce the notion of the \emph{magnetic Euler–Arnold equation}. This framework allows us to interpret several PDEs from fluid dynamics as magnetic Euler–Arnold equations. These include, for example, the Korteweg–de Vries equation~\eqref{e: KdV}, the generalized Camassa–Holm equation~\eqref{e: generalized Camassa–Holm}, the infinite conductivity equation~\eqref{e: infinite conductivity equation}, and the global quasi-geostrophic equations~\eqref{eq: global quasi-geostrophic equations}. In other words, these equations describe the motion of a charged particle on an infinite-dimensional manifold under the influence of an external magnetic field.

We summarize these results in \Cref{tab:magnetic_Euler_Arnold_PDEs}, where each PDE is associated with the magnetic system for which it is realized as a magnetic geodesic equation. In addition, we explicitly identify the corresponding Lorentz force, which represents the physical perturbation induced by the external magnetic field.
\begin{table}[H]
\centering
\begin{tabular}{|p{4.2cm}|p{5cm}|p{5cm}|}
\hline
\textbf{PDE} & \textbf{Magnetic system} & \textbf{Lorentz force} \\
\hline
\textit{Korteweg--de Vries equation~\eqref{e: KdV}} & 
\textit{$\mathrm{Diff}(\mathbb{S}^1)$ with $L^2$-metric and Gelfand--Fuchs cocycle (see~\Cref{c: KdV as magnetic geodesic equation})} & 
\textit{Dispersion term $a \cdot u_{xxx}$ (see~\Cref{r: viscosity in KdV is Lorenz force})} \\
\hline
\textit{Generalized Camassa--Holm equation~\eqref{e: generalized Camassa–Holm}} & 
\textit{$\mathrm{Diff}(\mathbb{S}^1)$ with $H^1$-metric and Gelfand--Fuchs cocycle (see~\Cref{C: gCH as an magnetic Euler-Arnold equation})} & 
\textit{Dispersion term (see~\Cref{r: Dispersion term in gCH})} \\
\hline
\textit{Infinite conductivity equation~\eqref{e: infinite conductivity equation}} & 
\textit{$\mathrm{Diff}_\mathrm{vol}(M)$ with $L^2$-metric and Lichnerowicz cocycle (see~\Cref{c: infinite conductivity equation as magn Euler-Arnold equation})} & 
\textit{Magnetic term $B \times u$ (see~\Cref{r: Lorenz forc in IC is magnetic drift term})} \\
\hline
\textit{Global quasi-geostrophic equations on a two-sphere~\eqref{eq: global quasi-geostrophic equations}} & 
\textit{Quantomorphism group of the 3-sphere $\mathbb{S}^3$ with right-invariant metric and trivial cocycle (see~\Cref{c: global quasi-geostrophic equations as magnetic})} & 
\textit{Correction term \( \frac{2z}{\mathrm{Ro}} + 2z h \) in~\cite{STS09, Ver09, LFEG24} (see~\Cref{r: lorenz force global quasi})} \\
\hline
\textit{Magnetic two-component Hunter--Saxton system (see~\cite[(M2HS)]{M24})} & 
\textit{Semidirect product group of diffeomorphisms and functions (see~\cite[Thm.~5.1]{M24})} & 
\textit{Rotation in an infinite-dimensional contact-type distribution (see~\cite[Eq.~5.3]{M24})} \\
\hline
\textit{Magnetic EPDiff equation (see~\cite[(MEpDiff)]{HopfRinowHalfLiegroups} and \cite[(MEpDiff)]{MaierRuscelliTonelli})} & 
\textit{Groups of diffeomorphisms on closed manifolds (see~\cite[(MEpDiff)]{HopfRinowHalfLiegroups, MaierRuscelliTonelli})} & 
\textit{Twist of the inertia operator (see~\cite[(MEpDiff)]{HopfRinowHalfLiegroups})} \\
\hline
\end{tabular}
\caption{Interpretation of selected PDEs as magnetic Euler--Arnold equations.}
\label{tab:magnetic_Euler_Arnold_PDEs}
\end{table}
Moreover, this framework allows us to interpret the Korteweg–de Vries equation~\eqref{e: KdV} as a magnetic deformation of the Burgers equation~\eqref{e: Burger}, which is the geodesic equation on $\mathrm{Diff}(\mathbb{S}^1)$ equipped with the $L^2$-metric. In this setting, the dispersion term in~\eqref{e: KdV} coincides precisely with the Lorentz force induced by the underlying infinite-dimensional magnetic system. A similar interpretation applies to the generalized Camassa–Holm equation~\eqref{e: generalized Camassa–Holm}, which may be viewed as a magnetic deformation of the Camassa–Holm equation~\eqref{e: Cammassa Holm eq}, the geodesic equation on $\mathrm{Diff}(\mathbb{S}^1)$ endowed with the $H^1$-metric.

In addition, the infinite conductivity equation~\eqref{e: infinite conductivity equation} can be interpreted as a magnetic deformation of the incompressible Euler equations. In this case, the magnetic term in~\eqref{e: infinite conductivity equation} coincides exactly with the Lorentz force defined by the associated magnetic system.

Finally, this framework enables us to interpret the global quasi-geostrophic equations\allowbreak\eqref{eq: global quasi-geostrophic equations} as an infinite-dimensional magnetic geodesic equation, where the correction term introduced in~\cite{STS09, Ver09, LFEG24} corresponds precisely to the Lorentz force of the associated magnetic system. Moreover, within this framework, we establish both local and global well-posedness for the magnetic Euler–Arnold equation corresponding to~\eqref{eq: global quasi-geostrophic equations}.
\ \\	\\
   \noindent
\textbf{Outlook:} 
A central theme of future interest is the exploration of similarities and differences between standard and magnetic geodesics. In finite-dimensional systems, the so-called Mañé critical value~\cite{Man} plays a fundamental role, serving as an energy threshold beyond which the dynamical and geometric properties of the magnetic geodesic flow typically change significantly. For finite-dimensional magnetic systems, the magnetic geodesic flow often resembles the standard geodesic flow above this threshold, whereas below it the behavior may differ substantially (see, for example,~\cite{ALBERS2025105521,Abbo13Lect,AbbMacMazzPat17,AssBenLust16,CIPP,CFP10,DMS25Contact,QuadraticgrowthLinaTom,Man,Merry2010}).

In~\cite{M24}, the author introduced a notion of the Mañé critical value for infinite-dimensional magnetic systems and illustrated it in~\cite[Thm.~7.1]{M24} using the magnetic two-component Hunter--Saxton system. It would therefore be natural to investigate whether the Mañé critical values associated with the magnetic systems listed in \Cref{tab:magnetic_Euler_Arnold_PDEs} provide new insights into the corresponding PDEs, particularly from the viewpoint of Hamiltonian dynamics.

Within the same differential-geometric framework, curvature offers another perspective. In the classical Euler--Arnold setting, the role of curvature—beginning with~\cite{Arnold66}, and especially its influence on the existence of conjugate points in diffeomorphism groups—has been studied extensively; see~\cite{AK98, Misiolek96conjtsTorus, MisiolekconjPtsKdV, Preston2006} and the references therein. A natural question is whether the recently introduced concept of magnetic curvature~\cite{Assenza2024} may play an analogous role in the magnetic Euler--Arnold setting. A recent finite-dimensional result~\cite{assenza2025conjugate} lends support to this perspective: the authors establish the existence of conjugate points along magnetic geodesics under suitable conditions on the magnetic curvature, suggesting that similar geometric phenomena might arise in the infinite-dimensional case as well.

We conclude this outlook by referring to \Cref{rem: speculative}, which raises the question of whether viewing the equations~\eqref{eq: global quasi-geostrophic equations} through the lens of exact magnetic systems and their associated action functionals could provide a fruitful approach to the study of measure-valued solutions—by analogy with the incompressible Euler equations, as explored in~\cite{DaneriFigalli2013} and the references therein.
\ \\	\\
   \noindent
\textbf{Structure of the paper:} 
In \Cref{S: 2}, we introduce the basic notions of magnetic systems and review fundamental concepts related to regular Lie groups. This provides the foundation for \Cref{d: magnetic euler equation}, where we define the magnetic Euler–Arnold equation for a magnetic system consisting of a regular Lie group equipped with a right-invariant Riemannian metric and a right-invariant closed two-form representing the magnetic field. In \Cref{t: magnetic Euler equation on g}, we prove that a curve is a magnetic geodesic of this system if and only if it satisfies the corresponding magnetic Euler–Arnold equation. This equation, which can be expressed in terms of the adjoint operator and the Lorentz force, constitutes the main theoretical contribution of the paper. Finally, in \Cref{ss: 2.3}, we relate our results to existing work in the literature.

In \Cref{S: 3}, we show that solutions of the magnetic Euler–Arnold equation correspond one-to-one with solutions of the Euler–Arnold equation on a central extension of the Lie group determined by the magnetic field, which defines a Lie algebra two-cocycle, as established in \Cref{C: relation magnetic geodesics and geodesics on central extension}. This correspondence holds only if the central extension of the Lie algebra integrates to a central extension of the Lie group—a condition that is not guaranteed in general.

In \Cref{s: 4}, we illustrate \Cref{t: magnetic Euler equation on g} by proving that the Korteweg–de Vries equation \eqref{e: KdV} and the generalized Camassa–Holm equation \eqref{e: generalized Camassa–Holm} are infinite-dimensional magnetic geodesic equations. We show that they can be viewed as magnetic deformations of the Burgers equation~\eqref{e: Burger} and the Camassa–Holm equation~\eqref{e: Cammassa Holm eq}, respectively, with the dispersion terms interpreted as infinite-dimensional Lorentz forces.

In \Cref{s: 5}, we prove that the infinite conductivity equation \eqref{e: infinite conductivity equation} is likewise an infinite-dimensional magnetic geodesic equation. We demonstrate how it may be understood as a magnetic deformation of the incompressible Euler equations~\eqref{e: Euler hydrp}, again interpreting the magnetic term as an infinite-dimensional Lorentz force.

We conclude the paper in \Cref{s: 6} by proving that the global quasi-geostrophic equations \eqref{eq: global quasi-geostrophic equations} are magnetic geodesic equations on the quantomorphism group. We interpret the correction term therein as an infinite-dimensional Lorentz force. Moreover, following the arguments in~\cite{ModinSurin2025}, we establish both local and global well-posedness for the magnetic Euler–Arnold equation associated with \eqref{eq: global quasi-geostrophic equations}.

\ \\
\noindent
\textbf{Acknowledgments:} The author is grateful to B.~Khesin for valuable comments and his continued interest in this work. In addition, the author thanks L.~Deschamps for carefully proofreading parts of the manuscript. The author also acknowledges M. Bauer, P. Michor, K. Modin, S. Preston, C. Vizman, and T. Diesz for helpful discussions. The author is indebted to the referees for their comments, which helped improve the paper.

This work was supported by the Deutsche Forschungsgemeinschaft (DFG, German Research Foundation) under grants 281869850 (RTG 2229), 390900948 (EXC-2181/1), and 281071066 (TRR 191). The author gratefully acknowledges the excellent working conditions and stimulating atmosphere at the Erwin Schrödinger International Institute for Mathematics and Physics in Vienna, during the thematic programme \emph{``Infinite-dimensional Geometry: Theory and Applications''}, where part of this work was completed.

    \section{Magnetic systems on regular Lie groups}\label{S: 2}
The aim of this subsection is to introduce an analogue of the Euler--Arnold equation for right-invariant magnetic systems—the magnetic Euler--Arnold equation—and to establish its equivalence with the corresponding equations of motion. The presentation follows the outline given in the introduction.
	\subsection{Intermezzo: Magnetic systems}\label{S: 2_Magnetic}

In the 1960s, the motion of a charged particle in a magnetic field was placed within the framework of modern dynamical systems by V.~Arnold in his pioneering work \cite{ar61}. The motion admits the following mathematical description:

\begin{defn}\label{d: magnetic system}
Let $(M,g)$ be a connected Riemannian tame Fréchet manifold, and let $\sigma \in \Omega^2(M)$ be a closed two-form. The form $\sigma$ is called a \emph{magnetic field}, and the triple $(M,g,\sigma)$ is called a \emph{magnetic system}. This structure defines a skew-symmetric bundle endomorphism $Y \colon TM \to TM$, called the \emph{Lorentz force}, by
\begin{equation}\label{e:Lorentz}
	g_q\left(Y_q u, v\right) = \sigma_q(u, v), 
	\qquad \forall\, q \in M,\ \forall\, u, v \in T_q M.
\end{equation}
A smooth curve $\gc \colon I\subseteq \mathbb{R} \to M$ is called a \emph{magnetic geodesic of strength} $s \in \mathbb{R}$ of the magnetic system $(M,g,\sigma)$ if it satisfies
\begin{equation}\label{e:mg_alt}
	\nabla_{\dot{\gc}} \dot{\gc} = s\, Y_{\gc} \dot{\gc},
\end{equation}
where $\nabla$ denotes the Levi--Civita connection associated with the metric $g$. A magnetic geodesic with strength $s = 1$ is simply referred to as a magnetic geodesic.
\end{defn}

\begin{rem}\label{r: duality magnetic geodesics of strength s and scaling magnetic form}
It follows directly from \Cref{d: magnetic system} that a curve $\gc$ is a magnetic geodesic of strength $s$ of $(M,g,\sigma)$ if and only if $\gc$ is a magnetic geodesic of $(M,g,s\cdot \sigma)$.
\end{rem}

From \eqref{e:mg_alt}, it follows that a magnetic geodesic with $s = 0$ reduces to a standard geodesic of the metric $g$. Therefore, \eqref{e:mg_alt} can be interpreted as a linear perturbation of the geodesic equation. The key point of interest is to explore the similarities and differences between standard and magnetic geodesics. 

Since $Y$ is skew-symmetric, magnetic geodesics have constant kinetic energy 
\[
E(\gamma,\dot\gamma):=\tfrac12 g_\gamma(\dot\gamma,\dot\gamma),
\]
and hence constant speed 
\[
|\dot\gamma|:=\sqrt{g_\gamma(\dot\gamma,\dot\gamma)},
\]
just as standard geodesics do. Energy conservation is a manifestation of the Hamiltonian nature of the system.

Indeed, we define the \emph{magnetic geodesic flow} on the tangent bundle by
\[
\varPhi_{g,\sigma}^t\colon TM\to TM,\quad 
(q,v)\mapsto \left( \gc_{q,v}(t),\dot{\gc}_{q,v}(t)\right),
\quad \forall t\in I\subseteq\mathbb{R},
\]
where $\gc_{q,v}(t)$ is the unique magnetic geodesic with initial values $(q,v)\in TM$. This flow admits the following Hamiltonian interpretation.

\begin{lem}[\cite{Gin}]\label{l: Ham form of magnetic flow on tangent bundle}
The magnetic geodesic flow $\varPhi^t_{g,\sigma}$ of $(M,g,\sigma)$ is the Hamiltonian flow induced by the kinetic energy $E\colon TM\to\mathbb{R}$ and the twisted symplectic form
\[
\omega_\sigma:=\mathrm{d}\lambda-\pi^*_{TM}\sigma,
\]
where $\lambda$ is the metric pullback of the canonical Liouville $1$-form from $T^*M$ to $TM$, and $\pi_{TM}\colon TM\to M$ is the canonical projection.
\end{lem}

From this perspective, the magnetic geodesic flow on $TM$ can be viewed as a deformation of the geodesic flow, obtained by modifying the underlying geometric structure—namely, by deforming the canonical symplectic structure $\mathrm{d}\lambda$ into the twisted symplectic structure $\omega_{\sigma}$. 

 \subsection{The magnetic Euler--Arnold equation on regular Lie groups}
Let us begin by introducing the setting in detail. 
Let $G$ be a regular Lie group in the sense of Kriegl--Michor~\cite{KrieglMichor1997}, with Lie algebra $\mathfrak{g}=T_{\id}G$, equipped with an inner product $\langle \cdot, \cdot \rangle_{\mathfrak{g}}$ and Lie bracket $[\cdot, \cdot]$. 
Throughout this paper, the Lie bracket on $\mathfrak{g}$ is defined via right-invariant vector fields on $G$. 
The inner product induces a right-invariant Riemannian metric $\mathcal{G}$ on $G$, defined by
\[
\mathcal{G}_\gamma(u,v)
:= \left\langle u\circ \gamma^{-1},\, v\circ\gamma^{-1}\right\rangle_{\mathfrak{g}},
\qquad u,v \in T_\gamma G.
\]
Let $\sigma\in \Omega^2(G)$ be a right-invariant closed two-form. We call the triple $(G,\mathcal{G},\sigma)$ a \emph{right-invariant magnetic system}.

Following Arnold’s approach~\cite{Arnold66, AK98}, which interprets the geodesic equation on a Lie group with a right-invariant metric as an evolution equation on its Lie algebra, we adapt this viewpoint to the magnetic setting.

Let $\gamma$ be a magnetic geodesic of $(G,\mathcal{G},\sigma)$. Its velocity $\dot{\gamma}$ is a tangent vector in $T_{\gamma}G$. Transporting it to the identity via right translation yields the \emph{Eulerian velocity}
\[
u:=\dot{\gamma}\circ\gamma^{-1}\in T_{\id}G=\mathfrak{g}.
\]
In this way, the magnetic geodesic equation on $G$ becomes an evolution equation on the Lie algebra $\mathfrak{g}$ for the curve $u$, given by a (generally nonlinear) dynamical system
\[
\dot{u}=F(u)
\]
on $\mathfrak{g}$. This motivates the following definition.

\begin{defn}\label{d: magnetic euler equation}
Let $(G,\mathcal{G},\sigma)$ be a right-invariant magnetic system. 
The evolution equation on the Lie algebra $\mathfrak{g}$ satisfied by the Eulerian velocity
\[
u=\dot{\gamma}\circ\gamma^{-1}
\]
of a magnetic geodesic $\gamma$ is called the \emph{magnetic Euler--Arnold equation} associated with $(G,\mathcal{G},\sigma)$.
\end{defn}
\begin{rem}
By choosing $\sigma=0$ in \Cref{d: magnetic euler equation}, we recover the classical notion of the Euler--Arnold equation corresponding to the metric $\mathcal{G}$ on $G$; see \cite[Def.~2.6]{km02}.
\end{rem}

In order to relate the magnetic Euler--Arnold equation in the sense of \Cref{d: magnetic euler equation} to the magnetic geodesic equation \eqref{e:Lorentz} of the magnetic system $(G,\mathcal{G},\sigma)$, we introduce the following notation. Denote by $(\cdot,\cdot)$ the natural pairing between $\mathfrak{g}$ and $\mathfrak{g}^{*}$. Following \cite{AK98}, we define
\begin{equation}\label{e: def inertia operator}
A: \mathfrak{g}\longrightarrow \mathfrak{g}^{*}, 
\qquad 
u \mapsto (u,\cdot),
\qquad \text{i.e.} \qquad 
(Au,v)=\langle u,v \rangle_{\mathfrak{g}} 
\quad \forall\, u,v\in \mathfrak{g},
\end{equation}
and call $A$ the \emph{inertia operator}.
 Recall that the coadjoint action $\mathrm{ad}^*$ of $\mathfrak{g}$ on $\mathfrak{g}^*$ is defined by
\begin{equation}\label{e: def coadjoint action}
(\mathrm{ad}^*_u(m),v)= (m,-\mathrm{ad}_u(v)) 
\quad \forall\, m\in \mathfrak{g}^*, \; u,v\in \mathfrak{g},
\end{equation}
where $\mathrm{ad}_u(v)=[u,v]$.

Before stating the next lemma, recall that the Euler--Arnold equation on $\mathfrak{g}^*$ can be naturally derived as a Hamiltonian equation using the canonical Lie--Poisson structure $\{\cdot,\cdot\}_{LP}$ on $\mathfrak{g}^*$; see, for example, \cite[§3]{km02}. This Lie--Poisson structure is induced by the canonical symplectic structure on $T^*G$ via symplectic reduction; see, for example, \cite{Marsden2007}.

Following this line of reasoning, deforming the canonical symplectic structure on $T^*G$ into the twisted symplectic structure yields, by \cite[Thm.~7.2.1]{Marsden2007}, a deformation of the Lie--Poisson structure. We call the resulting bracket the \emph{magnetic Lie--Poisson structure} on $\mathfrak{g}^*$ associated with the magnetic system $(G,\mathcal{G},\sigma)$, and it is given by
\begin{equation}\label{e: magnetic Lie Poisson bracket}
    \{f,g\}_{\sigma}(m)
    = \left(m, \left[\restr{\d f}{m}, \restr{\d g}{m}\, \right]\right)
      + \sigma_{\id}\big(\restr{\d f}{m}, \restr{\d g}{m}\big)
    \quad \forall\, m\in \mathfrak{g}^*,\; f,g\in C^{\infty}(\mathfrak{g}^*).
\end{equation}
Here 
\[
\{f,g\}_{LP}(m):=\left(m, \left[\restr{\d f}{m}, \restr{\d g}{m}\, \right]\right)
\]
is the classical Lie--Poisson bracket of $f$ and $g$ at the momentum $m\in \mathfrak{g}^*$.The Hamiltonian equation with respect to the magnetic Lie--Poisson structure~\eqref{e: magnetic Lie Poisson bracket} is as follows.

\begin{prop}\label{p: Hamiltonian vector field magnetic system Lie group}
The Hamiltonian vector field $X_f$ corresponding to a Hamiltonian $f\in C^{\infty}(\mathfrak{g}^*)$ with respect to the magnetic Lie--Poisson structure of $(G,\mathcal{G},\sigma)$ is given by
\[
X_f(m)= -\mathrm{ad}^{*}_{\d f}(m)+ A\left(Y_{\id}(\restr{\d f}{m}) \right)
\quad \forall m\in \mathfrak{g}^*,
\]
where $Y$ is the Lorentz force of the magnetic system $(G,\mathcal{G}, \sigma)$ in the sense of \eqref{e:Lorentz}. 

Thus, a curve $t\mapsto m(t)$ is a flow line of $X_f$ if and only if it satisfies the equation of motion
\[
\dot{m}= -\mathrm{ad}^{*}_{\d f}(m)+ A\left(Y_{\id}(\restr{\d f}{m}) \right).
\]
\end{prop}

\begin{rem}
Choosing $\sigma = 0$, we recover \cite[Prop.~3.2]{km02}. 
The difference in the sign in front of $\mathrm{ad}^{*}$ is due to the fact that we define the Lie bracket on $\mathfrak{g}$ via right-invariant vector fields, whereas \cite{km02} uses the left-invariant convention. This results in a sign change in the Lie bracket and consequently in $\mathrm{ad}^{*}$.
\end{rem}

\begin{proof}
Denote the Hamiltonian vector field of $f$ with respect to $\{\cdot,\cdot\}_{\sigma}$ by $X_f$. By definition and \eqref{e: magnetic Lie Poisson bracket}, for any function $g\in C^{\infty}(\mathfrak{g}^{*})$ we have
\begin{equation}\label{e: step1 proof Ham vector field Lie group}
\restr{ \d g(X_f)}{m}
= \{f,g\}_{\sigma}(m)
= (m, [\d f, \d g])
  + \sigma_{\id}(\d f,\d g)
= (-\mathrm{ad}^{*}_{\d f}(m), \d g)
  + \sigma_{\id}\left(\restr{\d f}{m},\restr{\d g}{m}\right).
\end{equation}
Using \eqref{e:Lorentz} and \eqref{e: def inertia operator}, we obtain
\[
\sigma_{\id}\left(\restr{\d f}{m},\restr{\d g}{m}\right)
= \mathcal{G}_{\id}\left(
Y_{\id}\left(\restr{\d f}{m}\right),
\restr{\d g}{m}
\right)
= \left(
A\circ Y_{\id}\left(\restr{\d f}{m}\right),
\d g
\right).
\]
Combining this with \eqref{e: step1 proof Ham vector field Lie group} yields the claimed formula for $X_f$.
\end{proof}
From \Cref{l: Ham form of magnetic flow on tangent bundle} and \eqref{e: def inertia operator}, we obtain that the Hamiltonian of the magnetic geodesic flow of $(G,\mathcal{G},\sigma)$, restricted to $\mathfrak{g}^*$, is given by
\[
H(m)=\frac{1}{2}(A^{-1}m,m),
\]
where $m=Au$. Consequently, using \Cref{p: Hamiltonian vector field magnetic system Lie group} and the identity $\restr{\d H}{m}=A^{-1}m$, we obtain the following.

\begin{cor}\label{c: magnetic euler equation dual of g}
The magnetic Euler--Arnold equation of $(G,\mathcal{G},\sigma)$ on $\mathfrak{g}^*$ for a curve $t\mapsto m(t)$ reads
\[
\dot{m}
= -\mathrm{ad}^{*}_{A^{-1}(m)}(m)
  - A\left(Y_{\id}(A^{-1}m)\right).
\]
\end{cor}
To derive the magnetic Euler--Arnold equation associated with $(G,\mathcal{G},\sigma)$, we assume that the adjoint $\mathrm{ad}^T$\footnote{Note that the corresponding notion in \cite{AK98} is $B(u,v)=\mathrm{ad}^T_u(v)$. } of $\mathrm{ad}$ with respect to $\langle \cdot, \cdot\rangle_{\mathfrak{g}}$ exists, i.e., that
\begin{equation}\label{e: definition ad^T}
    \langle \mathrm{ad}^T_u(v),w \rangle_{\mathfrak{g}}
    = \langle v, [u,w]\rangle_{\mathfrak{g}}
    \quad \forall\, u,v,w \in \mathfrak{g}.
\end{equation}
From \eqref{e: def inertia operator}, \eqref{e: def coadjoint action}, and \eqref{e: definition ad^T}, we obtain
\begin{equation}\label{e: relation ad T and ad*}
    \mathrm{ad}^*_u(Av)= -A(\mathrm{ad}^T_u(v))
    \quad \forall\, u,v\in \mathfrak{g}.
\end{equation}
For $m=Au$, we can therefore conclude from \Cref{c: magnetic euler equation dual of g} and \eqref{e: relation ad T and ad*} the main theoretical result of this note:

\begin{thm}[\textbf{The magnetic Euler--Arnold equation}]\label{t: magnetic Euler equation on g}
The curve $\gc$ is a magnetic geodesic of $(G,\mathcal{G},\sigma)$ if and only if 
\[
u:=\dot{\gc}\circ \gc^{-1}
\]
is a solution of the magnetic Euler--Arnold equation on $\mathfrak{g}$ associated with $(G,\mathcal{G},\sigma)$, namely
\[
\dot{u}=-\mathrm{ad}^T_u(u)- Y_{\id}(u).
\]
\end{thm}
\begin{rem}
The equivalence established in \Cref{t: magnetic Euler equation on g} 
provides the geometric starting point for the analysis of the magnetic 
EPDiff equation carried out in~\cite[§1.2]{HopfRinowHalfLiegroups}. 
In that work, the magnetic EPDiff equation is introduced via 
~\Cref{t: magnetic Euler equation on g} as the Eulerian formulation 
of the magnetic geodesic flow on Sobolev diffeomorphism groups.

Global magnetic geodesic completeness proved in~\cite[Thm.~1.6]{HopfRinowHalfLiegroups}, 
combined with a no gain–no loss argument for the associated flow, 
yields global well-posedness of the magnetic EPDiff equation 
for Sobolev indices $s > \frac{\dim M}{2} + 1$.
\end{rem}
\begin{rem}
For $\sigma=0$, we recover the geodesic equation on $(G,\mathcal{G})$ in its Eulerian form on $\mathfrak{g}$; see \cite{Arnold66, AK98}.
\end{rem}
\begin{rem}
Using \Cref{r: duality magnetic geodesics of strength s and scaling magnetic form}, we conclude from \Cref{t: magnetic Euler equation on g} that $\gc$ is a magnetic geodesic in $(G,\mathcal{G},a\cdot\sigma)$ if and only if
\[
u:=\dot{\gc}\circ \gc^{-1}
\]
is a solution of
\[
\dot{u}=-\mathrm{ad}^T_u(u)- a\cdot Y_{\id}(u).
\]
\end{rem}
  \subsection{Related results}\label{ss: 2.3}

We close this section by discussing developments connected to \Cref{t: magnetic Euler equation on g} that have previously appeared in the literature.

We begin by recalling \cite[Thm.~7.2.1]{Marsden2007}, which is based on infinite-dimensional symplectic reduction. This theorem gives rise to the \emph{magnetic Lie--Poisson structure} in~\eqref{e: magnetic Lie Poisson bracket}, which plays a key role in establishing the equivalence in \Cref{t: magnetic Euler equation on g} between the magnetic Euler--Arnold equation and V.~Arnold's formulation of magnetic systems in \cite{ar61}.

Next, we explain the relation between \Cref{t: magnetic Euler equation on g} and the \emph{Euler--Poincaré} equations introduced in~\cite{HolmMardsenRatiuEulerPoincare}. When the magnetic field in \Cref{t: magnetic Euler equation on g} is exact, the magnetic geodesic flow admits a Lagrangian formulation, as shown in~\cite{Gin}. In this case, the flow coincides with the Euler--Poincaré equations for right-invariant Lagrangians discussed in~\cite[Thm.~1.2]{HolmMardsenRatiuEulerPoincare}.

For general magnetic fields, however, no global primitive exists—even on a suitable covering space—as illustrated by simple systems such as the two-sphere. Consequently, the magnetic Euler--Arnold equation does not, in general, admit a Lagrangian formulation and is not governed by an action principle. It therefore differs from the Euler--Poincaré equations of~\cite{HolmMardsenRatiuEulerPoincare}.

We conclude this subsection by discussing the relation presented in \Cref{C: relation magnetic geodesics and geodesics on central extension} between the magnetic Euler--Arnold equation on a Lie group and the Euler--Arnold equation on a central extension of that group, where the magnetic field is interpreted as a closed two-cocycle on the Lie algebra. Such a central extension exists provided an integrability condition on the Lie algebra extension is satisfied; we refer to \Cref{S: 3} for a detailed discussion.

This Lie group extension exists, for example, when the Lie group $G$ arises via symplectic reduction of a smooth manifold $Q$ with respect to an $\SS^1=\mathbb{T}$-action. In that case, applying \Cref{C: relation magnetic geodesics and geodesics on central extension} together with \cite[Rmk.~4.2, Thm.~4.3]{ChengKhesinAveragingNonlin}, one recovers \Cref{t: magnetic Euler equation on g}.

We emphasize, however, that for a general Lie group $G$ equipped with a right-invariant closed two-form $\sigma$, the existence of such a manifold $Q$ depends on whether the corresponding integrability condition on $\sigma$ holds, which is not guaranteed in general. For instance, if $\sigma$ is an integral symplectic form, then this condition is satisfied via the Boothby--Wang construction; see \cite[§7]{Gg08}.

    \section{The one-to-one correspondence}\label{S: 3}

In this section, we establish a geometric correspondence between magnetic geodesics on a regular Lie group and ordinary geodesics on a suitable one-dimensional central extension of the group. The aim is to show that the magnetic Euler--Arnold equation can be interpreted as the classical Euler--Arnold equation on a central extension, provided an integrability condition for the magnetic field is satisfied.

	\subsection{Geodesics on central extensions of regular Lie groups}
\label{s: central extension Lie groups}

Let $\sigma \in H^2(\mathfrak{g}, \mathbb{R})$ be a $2$-cocycle. Note that $H^2(\mathfrak{g}, \mathbb{R})$ can be identified with the space of cohomology classes of $\mathcal{G}$-right-invariant $2$-forms on $G$.

For simplicity, we use $\sigma$ interchangeably to denote both a $2$-cocycle on $\mathfrak{g}$ and a $\mathcal{G}$-invariant closed $2$-form on $G$, without explicitly distinguishing between them in what follows. 

For the convenience of the reader, we recall that the central extension of $\mathfrak{g}$ with respect to the cocycle $\sigma \in H^2(\mathfrak{g}, \RR)$ is defined as the semidirect product
\[
\hat{\mathfrak{g}} := \mathfrak{g} \rtimes_{\sigma} \RR,
\]
with Lie bracket of two elements $(u,a),(v,b) \in \hat{\mathfrak{g}}$ given by
\begin{equation}\label{eq: lie algebra structure on extension}
[(u,a), (v,b)] := \left([u,v], \sigma(u,v)\right).
\end{equation}

Suppose that the one-dimensional central extension $(\hat{G}, \hat{\mathcal{G}})$ of $G$ exists, with Lie algebra $\hat{\mathfrak{g}}$, and is equipped with a right-invariant Riemannian metric $\hat{\mathcal{G}}$ defined at the identity by
\begin{equation}\label{e: extension of riem metric to central extension}
\hat{\mathcal{G}}_{(\id,0)}\big((u,a), (v,b)\big)
:= \mathcal{G}_{\id}(u,v) + a b
\quad \forall (u,a),(v,b)\in \hat{\mathfrak{g}}.
\end{equation}

The existence of such an extension typically requires an integrability condition on the Lie algebra $\hat{\mathfrak{g}}$, which we do not address here; see, for example, \cite{Vi08}.

We are now in a position to state \cite[Cor.~2]{Vi08}. From this point onward, we use $\dot{u} = u_t$ and $\dot{\gc} = \gc_t$ interchangeably.

\begin{prop}[\cite{Vi08}]\label{p: geodesics on central extensions}
The curve $(\gamma, a)$ is a geodesic in $(\hat{G}, \hat{\mathcal{G}})$ if and only if $(u,a)$, with $u=\dgc\circ \gamma^{-1}$, is a solution of
\begin{equation}\label{e: geodesic equation central extension}
\begin{cases}
u_t = -\mathrm{ad}^T_u(u) - a k(u), \\
a_t = 0,
\end{cases}
\end{equation}
where $k:\mathfrak{g}\to \mathfrak{g}$ is the unique operator satisfying
\[
\langle k(u), v\rangle_{\mathfrak{g}} = \sigma(u,v)
\quad \forall u,v\in \mathfrak{g}.
\]
\end{prop}

\begin{rem}
\textbf{Constancy of $a$.}
The parameter $a$ in \eqref{e: geodesic equation central extension} remains constant since $a_t = 0$.
\end{rem}

\begin{rem}
\textbf{The operator $k$ as Lorentz force.}
The operator $k:\mathfrak{g}\to\mathfrak{g}$ extends to a right-invariant operator $k:TG\to TG$ due to the right invariance of $\langle \cdot, \cdot\rangle_{\mathfrak{g}}$ and the definition of $\sigma$. Comparing with \Cref{e:Lorentz}, we see that this extension coincides with the Lorentz force of the magnetic system $(G, \mathcal{G}, \sigma)$.
\end{rem}
	\subsection{The correspondence between magnetic geodesics on the Lie group and geodesics on the central extension}

We are now in a position to derive, from the main theoretical result of this note (\Cref{t: magnetic Euler equation on g}), the above-mentioned one-to-one correspondence. We keep the notation of \Cref{S: 2_Magnetic} and \Cref{s: central extension Lie groups}.

\begin{cor}\label{C: relation magnetic geodesics and geodesics on central extension}
The curve $(\gamma,a)$ is a geodesic in $(\hat{G}, \hat{\mathcal{G}})$ if and only if $\gamma$ is a magnetic geodesic of strength $a$ in $(G, \mathcal{G}, \sigma)$.
\end{cor}

\begin{proof}
This follows directly by comparing \Cref{t: magnetic Euler equation on g} and \Cref{p: geodesics on central extensions}.
\end{proof}

\begin{rem}
Since we do not impose any integrability assumptions on the cocycle $\sigma \in H^2(\mathfrak{g}, \RR)$ in the formulation of \Cref{t: magnetic Euler equation on g}, the result is more general than \Cref{C: relation magnetic geodesics and geodesics on central extension}, where the existence of the central extension $\hat{G}$ is required. In general, this amounts to an integrability condition for the Lie algebra extension $\hat{\mathfrak{g}} := \mathfrak{g} \rtimes_{\sigma} \RR$.
\end{rem}

As an illustration of \Cref{t: magnetic Euler equation on g} and \Cref{C: relation magnetic geodesics and geodesics on central extension}, we show in the following sections that several well-known PDEs in mathematical physics can be formulated as magnetic Euler--Arnold equations. These include the Korteweg--de Vries equation \eqref{e: KdV}, the generalized Camassa--Holm equation \eqref{e: generalized Camassa–Holm}, the infinite conductivity equation \eqref{e: infinite conductivity equation}, and the global quasi-geostrophic equations \eqref{eq: global quasi-geostrophic equations}.

    \section{Two shallow water equations as magnetic Euler--Arnold equations} \label{s: 4}
In this section we derive the Korteweg--de Vries equation \eqref{e: KdV} 
and the generalized Camassa--Holm equation 
\eqref{e: generalized Camassa–Holm} within the unified framework 
introduced in \Cref{p: magnegtic Euler--Arnold equation for H1--alpha--beta--metric}. 
We show that both equations arise as magnetic geodesic equations 
corresponding to suitable right-invariant metrics. 
Our approach is strongly inspired by the framework developed by 
Khesin and Misiolek in \cite[Thm.~2.3]{km02}.
\subsection{The \texorpdfstring{$H^1_{\alpha, \beta}$}{H1 alpha beta}--Euler--Arnold equation}

We begin by recalling the definition of the \( H^1_{\alpha, \beta} \)-metric, for real parameters \( \alpha, \beta \in \mathbb{R} \), on the group \( \Diff(\SS^1) \) of smooth diffeomorphisms of \( \SS^1 \). It is given by
\begin{equation} \label{e: H1alpha--beta metric on Diff}
\mathcal{G}^{H^1_{\alpha,\beta}}_{\mathrm{id}}(u,v) := \int_{\mathbb{S}^1} \alpha\, u v + \beta\, u_x v_x \, \mathrm{d}x, 
\quad \text{for all } u, v \in T_{\mathrm{id}}\Diff(\mathbb{S}^1) = \mathfrak{X}(\mathbb{S}^1).
\end{equation}

Before proceeding, we record the following remark.

\begin{rem}\label{rem: H1alpha-beta metric as generalization of H1 and L2 metric on Diff}
For \( \alpha = 1 \) and \( \beta = 0 \), the \( H^1_{\alpha, \beta} \)-metric reduces to the \( L^2 \)-metric \( \mathcal{G}^{L^2} \) on \( \Diff(\mathbb{S}^1) \). 
For \( \alpha = \beta = 1 \), it coincides with the standard \( H^1 \)-metric \( \mathcal{G}^{H^1} \) on \( \Diff(\mathbb{S}^1) \).
\end{rem}

The magnetic term we consider is given by the cocycle \( c_{\mathrm{GF}} \), known as the \emph{Gelfand--Fuchs cocycle}, defined by
\begin{equation} \label{e: Gelfand-Fuchs cocycle}
c_{\mathrm{GF}}(u,v) := \int_{\mathbb{S}^1} u\, v_{xxx}\, \mathrm{d}x, 
\quad \text{for all } u, v \in \mathfrak{X}(\mathbb{S}^1).
\end{equation}

For completeness, we recall that the smooth dual of \( \mathfrak{X}(\mathbb{S}^1) \) is
\[
\mathfrak{X}(\mathbb{S}^1)^* 
= \{ u\,\mathrm{d}x^2 \mid u \in C^\infty(\mathbb{S}^1) \},
\]
and the dual pairing between \( u\,\mathrm{d}x^2 \in \mathfrak{X}(\mathbb{S}^1)^* \) and 
\( v \in \mathfrak{X}(\mathbb{S}^1) \) is given by
\begin{equation} \label{e: dual paring vir and vir+}
\left(u\,\mathrm{d}x^2, v\right) 
= \int_{\mathbb{S}^1} u\, v\, \mathrm{d}x.
\end{equation}

We are now in a position to state the first application of \Cref{t: magnetic Euler equation on g}.

\begin{prop}\label{p: magnegtic Euler--Arnold equation for H1--alpha--beta--metric}
Let \( a \in \mathbb{R} \) be fixed. 
The magnetic Euler--Arnold equation of the magnetic system 
\( \left(\Diff(\SS^1), \mathcal{G}^{H^1_{\alpha,\beta}}, a \cdot c_{\mathrm{GF}} \right) \) 
is given by
\begin{equation}\label{e: H1--alpha--beta magnetci EA-eq wrt GF cocycle}
\alpha \left( u_t + 3u u_x \right) 
- \beta \left( u_{txx} + 2 u_x u_{xx} + u u_{xxx} \right) 
= a u_{xxx}.
\end{equation}
In particular, \( \gc \) is a magnetic geodesic in 
\( \left(\Diff(\SS^1), \mathcal{G}^{H^1_{\alpha,\beta}}, a \cdot c_{\mathrm{GF}} \right) \) 
if and only if \( (u, a) \), with \( u = \gc_t \circ \gc^{-1} \), 
is a solution of \eqref{e: H1--alpha--beta magnetci EA-eq wrt GF cocycle}.
\end{prop}
\begin{proof}
To provide further insight into the result, we compute the magnetic Euler--Arnold equation of 
\( \left(\Diff(\SS^1), \mathcal{G}^{H^1_{\alpha,\beta}}, a\cdot c_{GF} \right) \) 
from first principles.

The inertia operator \( A \) of 
\( \left(\Diff(\SS^1),  \mathcal{G}^{H^1_{\alpha,\beta}} \right) \), 
with respect to the dual pairing on \( T_{\id}\Diff(\SS^1) \) induced by 
\eqref{e: dual paring vir and vir+}, is, by a computation analogous to the proof of 
\cite[Thm 3.6]{km02}, given by
\begin{equation}\label{e: inerta operator H1--alpha--beta--metric on Diff}
    A_{\alpha,\beta}(u)= \alpha u - \beta \partial_x^2 u 
    \quad \forall u\in T_{\id}\Diff(\SS^1).
\end{equation}

From \eqref{e: inerta operator H1--alpha--beta--metric on Diff}, 
\eqref{e: def inertia operator}, and \eqref{e:Lorentz}, we obtain that
\begin{equation}\label{e:lorenz force for H1--alpha--beta metric}
    Y_{\id}\colon T_{\id}\Diff(\SS^1) \longrightarrow T_{\id}\Diff(\SS^1),
    \quad u \mapsto -A_{\alpha,\beta}^{-1}(u_{xxx})
\end{equation}
is the Lorentz force evaluated at the identity for the magnetic system 
\( \left(\Diff(\SS^1), \mathcal{G}^{H^1_{\alpha,\beta}}, c_{GF} \right) \).

Moreover, by a computation following the lines of \cite[§18]{Vi08}, we obtain
\begin{equation}\label{e: ad-T for H1--alpha--beta metric on Diff}
        \mathrm{ad}^T_u(u)
        = A_{\alpha,\beta}^{-1}\left( 
        \alpha \cdot 3u u_{x}
        - \beta \cdot 2 u_x u_{xx}
        - \beta \cdot u u_{xxx}
        \right)
        \quad \forall u\in T_{\id}\Diff(\SS^1).
\end{equation}

Inserting \eqref{e:lorenz force for H1--alpha--beta metric} and 
\eqref{e: ad-T for H1--alpha--beta metric on Diff} into 
\Cref{t: magnetic Euler equation on g}, we conclude that
\begin{equation}\label{e: magnetic EA-eq of H1--alpha--beta and GF-cocyle}
      u_t
      = - A_{\alpha,\beta}^{-1}\left( 
        \alpha \cdot 3u u_{x}
        - \beta \cdot 2 u_x u_{xx}
        - \beta \cdot u u_{xxx}
        \right)
        + a A_{\alpha,\beta}^{-1}(u_{xxx})
\end{equation}
is the magnetic Euler--Arnold equation of the magnetic system 
\( \left(\Diff(\SS^1), \mathcal{G}^{H^1_{\alpha,\beta}}, a\cdot c_{GF} \right) \).

Applying \( A_{\alpha,\beta} \) to both sides of 
\eqref{e: magnetic EA-eq of H1--alpha--beta and GF-cocyle} and using 
\eqref{e: inerta operator H1--alpha--beta--metric on Diff}, we obtain 
\eqref{e: H1--alpha--beta magnetci EA-eq wrt GF cocycle}, which completes the proof.
\end{proof}

\begin{rem}
Alternatively, \Cref{p: magnegtic Euler--Arnold equation for H1--alpha--beta--metric} 
can also be derived using \cite[Thm.~3.6]{km02}, which identifies 
\eqref{e: H1--alpha--beta magnetci EA-eq wrt GF cocycle} as the Euler--Arnold equation 
on the one-dimensional central extension \( \mathrm{Vir} \), the so-called Virasoro group, 
of \( \Diff(\SS^1) \) with respect to the Gelfand--Fuchs cocycle 
\eqref{e: Gelfand-Fuchs cocycle}, equipped with the extension of the 
\( H^1_{\alpha,\beta} \)-metric \eqref{e: H1alpha--beta metric on Diff} 
as described in \eqref{e: extension of riem metric to central extension}. 
Thus, applying \Cref{C: relation magnetic geodesics and geodesics on central extension} 
yields an alternative proof.
\end{rem}
\subsection{Korteweg--de Vries equation as magnetic Euler--Arnold equation}

The Korteweg--de Vries equation \eqref{e: KdV}, introduced by Korteweg--de Vries in \cite{korteweg1895change}, 
serves as a mathematical model for waves on shallow water surfaces and is given by
\begin{equation}\label{e: KdV}
 u_t=-3u u_{x}+a u_{xxx}\tag{KdV}
\end{equation}
for smooth $u:I\times \SS^1\longrightarrow \RR$, where $\SS^1=\RR/\ZZ$, 
$I=[0,T)$, and $a\in \RR$.

Choosing \( \alpha = 1 \) and \( \beta = 0 \) in 
\eqref{e: H1--alpha--beta magnetci EA-eq wrt GF cocycle}, 
the \( H^1_{\alpha,\beta} \)-metric reduces to the \( L^2 \)-metric, 
as noted in \Cref{rem: H1alpha-beta metric as generalization of H1 and L2 metric on Diff}. 
Substituting \( \alpha = 1 \) and \( \beta = 0 \) into 
\eqref{e: H1--alpha--beta magnetci EA-eq wrt GF cocycle}, 
we obtain from \Cref{p: magnegtic Euler--Arnold equation for H1--alpha--beta--metric} the following.

\begin{cor}\label{c: KdV as magnetic geodesic equation}
For fixed $a\in\RR$, the Korteweg--de Vries equation \eqref{e: KdV} is the magnetic Euler--Arnold equation of the magnetic system 
$\left( \Diff(\SS^1), \mathcal{G}^{L^2}, a\cdot c_{GF}\right)$. \\
In particular, a curve $\gamma$ is a magnetic geodesic in 
$\left( \Diff(\SS^1), \mathcal{G}^{L^2}, c_{GF}\right)$ of strength $a$ 
if and only if $(u,a)$ with $u:=\dot\gamma\circ \gamma^{-1}$ 
is a solution of \eqref{e: KdV}. 
\end{cor}

\begin{rem}
Alternatively, choosing $\alpha=1$ and $\beta=0$ and applying 
\Cref{C: relation magnetic geodesics and geodesics on central extension} 
in combination with \cite[Prop. 1]{OvsienkoKhesin1987} yields another proof. 
\end{rem}

\begin{rem}
\textbf{KdV as a magnetic deformation of the Burgers equation.}  The above result allows us to interpret \eqref{e: KdV} as a magnetic deformation 
of the so-called \emph{Burgers equation} in the following sense. 
Recall that $\gamma$ is a geodesic in 
$\left(\Diff(\SS^1), \mathcal{G}^{L^2}\right)$ 
if and only if $u=\dgc\circ \gamma^{-1}$ 
solves the Burgers equation
\begin{equation} \label{e: Burger}
	u_t+3uu_{x}=0, \tag{Burger}
\end{equation}
where we used $\mathrm{ad}^T_u(u)=3u u_x$ on 
$\left(\Diff(\SS^1), \mathcal{G}^{L^2}\right)$.\\
Using \eqref{e:Lorentz}, 
\Cref{rem: H1alpha-beta metric as generalization of H1 and L2 metric on Diff}, 
and \eqref{e: Gelfand-Fuchs cocycle}, together with three integrations by parts, 
the Lorentz force \( Y \) of the magnetic system 
\( \left( \Diff(\SS^1), \mathcal{G}^{L^2}, c_{\mathrm{GF}} \right) \) 
is given at the identity by
\begin{equation}\label{e: lorenz force for L2 metric and GF-cocyle}
Y \colon \mathfrak{X}(\SS^1) \longrightarrow \mathfrak{X}(\SS^1), 
\quad u \mapsto -u_{xxx}.
\end{equation}
The operator $a\cdot Y$ is precisely the difference between 
\eqref{e: KdV} and \eqref{e: Burger}. 
Thus, by turning on the magnetic field on $\Diff(\SS^1)$, that is, choosing $a\neq 0$, 
we deform \eqref{e: Burger} into \eqref{e: KdV} via the Lorentz force:
\[
\underbrace{u_t+3uu_{x}}_{= \nabla _{u}u }
=
\underbrace{a u_{xxx}}_{= -a\cdot Y_{\id}(u)}. 
\]
\end{rem}

\begin{rem}\label{r: viscosity in KdV is Lorenz force}
\textbf{Dispersion term in~\eqref{e: KdV} as a Lorentz force.} \Cref{c: KdV as magnetic geodesic equation} allows us to interpret 
the dispersive term in \eqref{e: KdV}, namely $a u_{xxx}$, 
as the Lorentz force acting on a charged particle in the infinite-dimensional 
magnetic system $\left( \Diff(\SS^1), \mathcal{G}^{L^2}, c_{GF}\right)$. 
\end{rem}

\subsection{Generalized Camassa--Holm equation as magnetic Euler--Arnold equation}
\label{s: gCH as an magnetic Euler-Arnold equation}

The following shallow water equation, which is a completely integrable nonlinear 
partial differential equation,
\begin{equation}\label{e: generalized Camassa–Holm}
 u_t-u_{txx}= -3 u u_x+ 2 u_xu_{xx}+u u_{xxx}+ a u_{xxx} \tag{gCH}
\end{equation}
is called the \emph{generalized Camassa--Holm equation}, introduced by 
Camassa--Holm in \cite{CamassaHolm1993}, with $(u,a)$ as in \eqref{e: KdV}.

Choosing \( \alpha = 1 = \beta \) in 
\eqref{e: H1--alpha--beta magnetci EA-eq wrt GF cocycle}, 
the \( H^1_{\alpha,\beta} \)-metric reduces to the \( H^1 \)-metric, 
as noted in \Cref{rem: H1alpha-beta metric as generalization of H1 and L2 metric on Diff}. 
Substituting \( \alpha = 1 = \beta \) into 
\eqref{e: H1--alpha--beta magnetci EA-eq wrt GF cocycle}, 
we obtain from \Cref{p: magnegtic Euler--Arnold equation for H1--alpha--beta--metric} 
the following.

\begin{cor}\label{C: gCH as an magnetic Euler-Arnold equation}
For fixed $a\in \RR$, the generalized Camassa--Holm equation 
\eqref{e: generalized Camassa–Holm} is the magnetic Euler--Arnold equation 
of the magnetic system 
$\left( \Diff(\SS^1), \mathcal{G}^{H^1}, a\cdot c_{GF}\right)$.\\
In particular, a curve $\gamma$ is a magnetic geodesic of
$\left( \Diff(\SS^1), \mathcal{G}^{H^1}, c_{GF}\right)$ of strength $a$ 
if and only if $(u,a)$ with $u:=\dgc\circ \gamma^{-1}$ 
is a solution of \eqref{e: generalized Camassa–Holm}. 
\end{cor}

\begin{rem}
Alternatively, choosing $\alpha=1$ and $\beta=1$ and applying 
\Cref{C: relation magnetic geodesics and geodesics on central extension} 
in combination with \cite[Thm. 1]{MISIOLEK1998203} 
yields another proof of \Cref{C: gCH as an magnetic Euler-Arnold equation}. 
\end{rem}

\begin{rem}\label{r: gCH as mag deform of Ch}
\textbf{\eqref{e: generalized Camassa–Holm} as a magnetic deformation of \eqref{e: Cammassa Holm eq}.} This result allows us to interpret \eqref{e: generalized Camassa–Holm} 
as a magnetic deformation of the so-called \emph{Camassa--Holm equation} 
\eqref{e: Cammassa Holm eq} in the following sense.\\
Recall that by \cite[Thm. IV.1]{kouranbaeva1999camassa}, 
a curve $\gamma$ is a geodesic in 
$\left(\Diff(\SS^1), \mathcal{G}^{H^1}\right)$ 
if and only if $u=\dgc\circ \gamma^{-1}$ 
is a solution of the Camassa--Holm equation
\begin{equation} \label{e: Cammassa Holm eq}
 u_t-u_{txx}= -3 u u_x+ 2 u_xu_{xx}+u u_{xxx}. \tag{CH}
\end{equation}
This formulation relies crucially on the identity
\begin{equation}\label{e: ad-T for H1 metric on Diff}
    \mathrm{ad}^T_u(u)
    =A^{-1}\left( 3u u_{x}-2 u_xu_{xx}-u u_{xxx}\right)
    \quad \forall u\in T_{\id}\Diff(\SS^1),
\end{equation}
where $A=1-\partial_x^2$ denotes the inertia operator of 
$\mathcal{G}^{H^1}$ on 
$\left(\Diff(\SS^1), \mathcal{G}^{H^1}\right)$ 
with respect to the dual pairing on 
$T_{\id}\Diff(\SS^1)$ induced by \eqref{e: dual paring vir and vir+}. From the definition of the inertia operator \eqref{e: def inertia operator}, 
together with \eqref{e:Lorentz} and 
\eqref{e: lorenz force for L2 metric and GF-cocyle}, we obtain
\begin{equation}\label{e: Lorenz force for H1 metric and GF cocyle}
Y: T_{\id}\Diff(\SS^1)\longrightarrow T_{\id}\Diff(\SS^1),
\quad u\mapsto -A^{-1}(u_{xxx}),
\end{equation}
which is the Lorentz force of the magnetic system 
$\left( \Diff(\SS^1), \mathcal{G}^{H^1}, c_{GF}\right)$. Thus, the magnetic Euler--Arnold equation of 
$\left(\Diff(\SS^1), \mathcal{G}^{H^1}, a\cdot c_{GF} \right)$ 
in the form of \Cref{t: magnetic Euler equation on g} reads
\begin{equation}\label{e: in rmk for gCH as mag. deform. of CH}
    u_t
    = - A^{-1}\left( 3u u_{x}-2 u_xu_{xx}-u u_{xxx}\right)
      + a A^{-1}(u_{xxx}).
\end{equation}
Applying $A$ to both sides of 
\eqref{e: in rmk for gCH as mag. deform. of CH} 
yields \eqref{e: generalized Camassa–Holm}. Moreover, the operator $aA\left(Y_{\id}\right)$ 
is precisely the difference between 
\eqref{e: generalized Camassa–Holm} 
and \eqref{e: Cammassa Holm eq}, 
as becomes clear by comparing the two equations.
\end{rem}

\begin{rem}\label{r: Dispersion term in gCH}
\textbf{Dispersion term in~\eqref{e: generalized Camassa–Holm} as a Lorentz force.} Similar to \Cref{r: viscosity in KdV is Lorenz force}, 
\Cref{C: gCH as an magnetic Euler-Arnold equation} 
allows us to interpret the dispersive term in 
\eqref{e: generalized Camassa–Holm}, namely $a u_{xxx}$, 
as the Lorentz force acting on a charged particle in the 
infinite-dimensional magnetic system 
\[\left( \Diff(\SS^1), \mathcal{G}^{H^1}, c_{GF}\right)\, .\]
\end{rem}	

\section{The infinite conductivity equation as a magnetic Euler--Arnold equation}\label{s: 5}
	The \emph{infinite conductivity equation} \eqref{e: infinite conductivity equation} 
models the motion of a high-density electronic gas in a magnetic field with given velocity 
on a three-dimensional closed orientable Riemannian manifold $(M,g)$. 

Before stating the \emph{infinite conductivity equation}, we introduce some notation. 
We denote by $\mathrm{vol}$ the Riemannian volume form induced by $g$ and the chosen orientation, 
and by $\nabla$ the Levi-Civita connection on $(M,g)$. 
Let $\eta \in \Omega^2(M)$ be a closed two-form. 
Since $\mathrm{vol}$ is nondegenerate, there exists a unique divergence-free vector field 
$B\in \mathfrak{X}_{\mathrm{vol}}(M)$ such that
\[
\iota_B \mathrm{vol} = -\eta,
\]
where $\mathfrak{X}_{\mathrm{vol}}(M)$ denotes the Lie algebra of divergence-free vector fields on $M$.

Since $M$ is oriented and three-dimensional, the metric $g$ and the volume form 
$\mathrm{vol}$ induce a well-defined cross product $\times$ on vector fields on $M$.

Thus, the \emph{infinite conductivity equation} in a magnetic field 
$B\in \mathfrak{X}_{\mathrm{vol}}(M)$ with velocity 
$u\in \mathfrak{X}(M)$ is
\begin{equation}\label{e: infinite conductivity equation}
\begin{cases}
u_t+\nabla_u u = - a \cdot B\times u+ \nabla p,\\
\mathrm{div}\, u = 0
\end{cases}
\tag{IC}
\end{equation}
where $\times$ denotes the cross product of vector fields on $M$, 
and $\nabla p$ denotes the gradient of a smooth function $p$ on $(M,g)$.

In order to interpret the infinite conductivity equation~\eqref{e: infinite conductivity equation} 
from a geometric perspective, we follow the formalism developed in~\cite{Vi08} 
and introduce the necessary notation.
Let \( \Diff_{\mathrm{vol}}(M) \) denote the group of volume-preserving diffeomorphisms 
of a Riemannian manifold \( (M, g) \) with respect to the volume form \( \mathrm{vol} \). 
Its Lie algebra is \( \mathfrak{X}_{\mathrm{vol}}(M) \), the space of divergence-free 
vector fields on \( M \), equipped with the Lie bracket defined as the negative 
of the usual Lie bracket of vector fields:
\[
\mathrm{ad}(u, v) = -[u, v].
\]

The \( L^2 \)-metric on \( \Diff_{\mathrm{vol}}(M) \), defined at the identity, is given by
\begin{equation} \label{e: L2 metric riemannain manifold}
\mathcal{G}^{L^2}_{\id}(u, v) 
= \int_{M} g(u, v)\, \mathrm{vol}, 
\qquad \forall u, v \in \mathfrak{X}_{\mathrm{vol}}(M),
\end{equation}
and it extends to a right-invariant Riemannian metric on the entire group 
\( \Diff_{\mathrm{vol}}(M) \).

This framework allows us to recall the seminal result of V.~Arnold~\cite{Arnold66}, 
which establishes the correspondence between geodesics on this infinite-dimensional 
Lie group and the classical Euler equations for incompressible fluids.

\begin{thm}[\cite{Arnold66}]
A curve \( \gamma(t) \) is a geodesic in 
\( \left(\Diff_{\mathrm{vol}}(M), \mathcal{G}^{L^2} \right) \) 
if and only if the Eulerian velocity field 
\( u := \dot{\gamma} \circ \gamma^{-1} \) 
satisfies the incompressible Euler equations:
\begin{equation} \label{e: Euler hydrp}
\begin{cases}
u_t+\nabla_u u & = - \nabla p,\\
\mathrm{div}\, u&=0
\end{cases}
\tag{Euler}
\end{equation}
for some pressure function \( p \in C^{\infty}(M) \).
\end{thm}

To incorporate magnetic effects into this picture, we consider a closed two-form 
\( \eta \) on \( M \), which gives rise to the so-called 
\emph{Lichnerowicz 2-cocycle} \( \Omega_{\eta} \) 
on \( \mathfrak{X}_{\mathrm{vol}}(M) \), defined by
\begin{equation} \label{e: def Lichnorowics cocyle.}
\Omega_{\eta}(u, v) 
:= \int_M \eta(u, v)\, \mathrm{vol}, 
\qquad \forall u, v \in \mathfrak{X}_{\mathrm{vol}}(M).
\end{equation}
With this structure in place, we are now in a position to state the following result.
\begin{cor} \label{c: infinite conductivity equation as magn Euler-Arnold equation}
Let \( a \in \mathbb{R} \) be fixed. 
Then the infinite conductivity equation~\eqref{e: infinite conductivity equation} 
arises as the magnetic Euler--Arnold equation associated with the magnetic system
\[
\left( \Diff_{\mathrm{vol}}(M), \mathcal{G}^{L^2}, a \cdot \Omega_{\eta} \right).
\]
In particular, a curve \( \gamma\) 
is a magnetic geodesic of this system if and only if the Eulerian velocity field 
\( u := \dot{\gamma} \circ \gamma^{-1} \) 
satisfies~\eqref{e: infinite conductivity equation}.
\end{cor}

\begin{proof}
By a computation along the lines of~\cite[§10]{Vi08}, 
the Lorentz force evaluated at the identity of the magnetic system 
\( \left( \Diff_{\mathrm{vol}}(M), \mathcal{G}^{L^2}, \Omega_{\eta} \right) \) 
is given by
\begin{equation} \label{e: Lorenz force super cond eq}
Y_{\mathrm{id}} \colon \mathfrak{X}_{\mathrm{vol}}(M) 
\longrightarrow \mathfrak{X}_{\mathrm{vol}}(M), 
\qquad u \mapsto B \times u.
\end{equation}
Moreover, as shown in~\cite{Arnold66}, 
the adjoint of \( \mathrm{ad} \) with respect to the \( L^2 \)-metric satisfies
\begin{equation} \label{eq: adjoint volum preserving}
\mathrm{ad}_u^T(u) = \nabla_u u + \nabla p.
\end{equation}
Combining \Cref{t: magnetic Euler equation on g} 
with \eqref{e: Lorenz force super cond eq} 
and \eqref{eq: adjoint volum preserving} 
yields the infinite conductivity equation 
\eqref{e: infinite conductivity equation}, 
which completes the proof.
\end{proof}

\begin{rem}
In contrast to the results in \cite{Vizman2001, Vi08}, 
for \Cref{c: infinite conductivity equation as magn Euler-Arnold equation} 
we may drop the topological restrictions involving the homology or homotopy groups of \( M \). 
Indeed, we do not require the central extension of the Lie algebra 
\( \mathfrak{X}_{\mathrm{vol}}(M) \), defined via the Lichnerowicz cocycle, 
to integrate to a Lie group extension of \( \Diff_{\mathrm{vol}}(M) \). \\
However, if one imposes the same topological assumptions on \( M \) 
as in \cite{Vizman2001, Vi08}, then 
\Cref{c: infinite conductivity equation as magn Euler-Arnold equation} 
also follows directly from 
\Cref{C: relation magnetic geodesics and geodesics on central extension} 
and the results of \cite{Vizman2001, Vi08}.
\end{rem}

\begin{rem}
\label{r: infinite conductivity as mag deform of Euler equation} 
\textbf{\eqref{e: infinite conductivity equation} as a magnetic deformation of the Euler equation.}  This result allows us to interpret 
\eqref{e: infinite conductivity equation} 
as a magnetic deformation of the Euler equation 
\eqref{e: Euler hydrp} from ideal hydrodynamics in the following sense.\\
The operator \( Y_{\mathrm{id}} \) 
is precisely the difference between 
\eqref{e: infinite conductivity equation} 
and \eqref{e: Euler hydrp}, 
as becomes apparent by comparing the two equations.
\end{rem}

\begin{rem}
\textbf{The magnetic force in \eqref{e: infinite conductivity equation} 
as an infinite-dimensional Lorentz force.} 
\label{r: Lorenz forc in IC is magnetic drift term} \Cref{c: infinite conductivity equation as magn Euler-Arnold equation} 
allows us to interpret the force generated by the magnetic field \( B \) 
in \eqref{e: infinite conductivity equation}, namely \( -B \times u \), 
as the Lorentz force acting on a charged particle in the infinite-dimensional 
magnetic system 
\( \left( \Diff_{\mathrm{vol}}(M), \mathcal{G}^{L^2}, \Omega_{\eta}\right) \).
\end{rem}
\section{The Global Quasi-Geostrophic Equations as a Magnetic Geodesic Equation}\label{s: 6}
We begin by recalling some background on the model. 
The quasi-geostrophic approximation for large-scale atmospheric 
and oceanic flows was originally introduced by Charney in 1949~\cite{Charney49}. 
When considering global fluid motion on the two-dimensional sphere 
\( \mathbb{S}^2 \subseteq \mathbb{R}^3 \) of radius \( 1/2 \), 
curvature effects must be taken into account. 
These geometric corrections, developed in~\cite{STS09, Ver09, LFEG24}, 
lead to the formulation of the \emph{global quasi-geostrophic equations}. 

This system models an incompressible, inviscid two-dimensional fluid 
in terms of the potential vorticity function 
\( q:\SS^2\times I\longrightarrow \RR \) 
and the stream function 
\( \psi:\SS^2\times I\longrightarrow \RR \):
\begin{equation}
\label{eq: global quasi-geostrophic equations}
\partial_t q + \{\psi, q\} = 0, \qquad 
q = (\Delta - \gamma z^2)\psi + \frac{2z}{\mathrm{Ro}} + 2zh. 
\tag{Global QG}
\end{equation}

Here, \( z = \cos(\vartheta) \), where 
\( \vartheta \in [-\pi, \pi] \) denotes the latitude, 
and \( h \) represents the bottom topography. 
The bracket \( \{ \cdot, \cdot \} \) denotes the standard Poisson bracket 
on \( \mathbb{S}^2 \). 
For the definitions and physical interpretations of the parameters 
\( \gamma \) and \( \mathrm{Ro} \), as well as further background, 
we refer to~\cite{ModinSurin2025} and the references therein.

\subsection{The geometric setting}To formulate equation~\eqref{eq: global quasi-geostrophic equations} 
as an Euler--Arnold equation, we first introduce some notation. 
Let $\mathbb{S}^3$ denote the three-sphere in $\mathbb{C}^2$, 
equipped with Hopf coordinates. 
In this parametrization, the complex coordinates $w_1$ and $w_2$ 
on $\mathbb{S}^3 \subset \mathbb{C}^2$ are given by
\[
w_1 = \cos \eta \, e^{\mathrm{i} \xi_1}, 
\qquad 
w_2 = \sin \eta \, e^{\mathrm{i} \xi_2},
\]
where $\eta \in \left(0, \frac{\pi}{2} \right)$ and 
$\xi_1, \xi_2 \in (0, 2\pi)$. The Euclidean metric on $\mathbb{C}^2$ induces the standard 
Riemannian metric on $\mathbb{S}^3$, which in these coordinates 
takes the form
\begin{equation}
\label{eq:metricS3}
g^{\Sd} 
= \cos^2 \eta \, \mathrm{d}\xi_1^2 
+ \sin^2 \eta \, \mathrm{d}\xi_2^2 
+ \mathrm{d}\eta^2.
\end{equation}
An orthogonal frame with respect to this metric is given by
\begin{equation}\label{eq: orthonormal frame}
R := \partial_{\xi_1} + \partial_{\xi_2}, 
\qquad 
E_2 := \partial_\eta, 
\qquad 
E_3 := \partial_{\xi_1} - \partial_{\xi_2}.
\end{equation}
Here, $R$ generates the Reeb flow 
$\varPhi^t(z) := e^{\mathrm{i}t} z$ 
and is tangent to the $\mathbb{S}^1$-fibers of the Hopf fibration
\[
\pi \colon \mathbb{S}^3 \longrightarrow \mathbb{S}^2.
\]
We denote by $\lambda$ the metric dual of $R$ 
with respect to the metric~\eqref{eq:metricS3}. 
This 1-form $\lambda$ is the standard contact form on $\mathbb{S}^3$. 
It is well known that $\lambda \wedge \mathrm{d}\lambda$ 
coincides with the Riemannian volume form 
$\mathrm{vol}$ associated with $g^{\mathbb{S}^3}$ as in \eqref{eq:metricS3}.\\
Moreover, $R$ is uniquely characterized by the conditions 
$\lambda(R) = 1$ and $\mathrm{d}\lambda(R, \cdot) = 0$. 
The vector fields $E_2$ and $E_3$ span the standard contact distribution 
$\xi := \ker \lambda$. 
For further background on contact geometry, we refer to~\cite{Gg08}.
\\
The \emph{quantomorphism group} of $(\mathbb{S}^3, \lambda)$ is defined as
\begin{equation}
\label{eq: defi quant group}
\mathcal{D}_q(\mathbb{S}^3) 
:= \mathcal{D}(\mathbb{S}^3, \lambda) 
:= \left\{ F \in C^\infty(\mathbb{S}^3, \mathbb{S}^3) 
\mid F^* \lambda = \lambda \right\},
\end{equation}
and is also known as the group of \emph{strict contactomorphisms} 
of $(\mathbb{S}^3, \lambda)$.\\
Let $C^\infty_R(\mathbb{S}^3)$ denote the space of smooth functions 
on $\mathbb{S}^3$ that are invariant under the Reeb flow, 
i.e., those satisfying $R(f) = 0$. 
Such functions can be naturally identified with smooth functions on $\mathbb{S}^2$. \\
As derived in~\cite[p.~20]{EP13}, this space gives rise to the differential operator
\begin{equation}
\label{eq: operator S}
S_\lambda f 
= f R 
- \tfrac{1}{2}(E_3 f) E_2 
+ \tfrac{1}{2}(E_2 f) E_3,
\end{equation}
where $S_\lambda f = u$ if and only if 
$\lambda(u) = f$ and $\iota_u \mathrm{d}\lambda = -\mathrm{d}f$ 
(see~\cite{EP13} for details).\\
With this notation, the Lie algebra $\mathfrak{g}$ of the quantomorphism group 
$\mathcal{D}_q(\mathbb{S}^3)$ can be identified as
\[
\mathfrak{g} 
= T_{\mathrm{id}} \mathcal{D}_q \left(\Sd\right)
= \left\{ S_\lambda f 
\mid f \in C^\infty_R(\mathbb{S}^3) \right\}.
\]
In the sequel, let $\rho \colon \mathbb{S}^3 \to \mathbb{R}$ 
be a smooth $\mathbb{S}^1$-invariant function. 
This function defines a differential operator
\[
A \colon \mathfrak{X}(\mathbb{S}^3) \to \mathfrak{X}(\mathbb{S}^3),
\]
called the \emph{inertia operator}, which acts on vector fields 
$u \in \mathfrak{X}(\mathbb{S}^3)$ by
\[
A(u) 
:= \rho^2 g^{\mathbb{S}^3}(u, R)\, R 
+ g^{\mathbb{S}^3}(u, E_2)\, E_2 
+ g^{\mathbb{S}^3}(u, E_3)\, E_3.
\]
Restricting $A$ to the Lie algebra $\mathfrak{g}$, 
it induces a positive-definite inner product
\[
\langle \cdot, \cdot \rangle_A 
\colon \mathfrak{g} \times \mathfrak{g} \to \mathbb{R},
\]
defined by
\begin{align}
\label{eq:inner_product_A}
\langle S_\lambda f,\, S_\lambda g \rangle_A
&= \int_{\mathbb{S}^3} 
g^{\mathbb{S}^3} \!\left(
f \rho^2 R 
- \tfrac{1}{2}(E_3 f) E_2 
+ \tfrac{1}{2}(E_2 f) E_3,\,
S_\lambda g
\right)
\, \mathrm{d}\mathrm{vol}.
\end{align}
This inner product defines a right-invariant weak Riemannian metric 
$\mathcal{G}^A$ on $\mathcal{D}_q(\mathbb{S}^3)$.\\
By~\cite[Prop.~2.3]{ModinSurin2025}, 
the adjoint of the operator $S_{\lambda}$ with respect to the weak 
Riemannian metric defined in~\eqref{eq:inner_product_A} exists and is denoted by 
$S^*_{\lambda, A}$. 
This adjoint gives rise to the \emph{contact Laplacian}, which, according to~\cite[Prop.~2.4]{ModinSurin2025}, is the operator
\begin{equation}
\label{eq: contact Laplacian}
\Delta_{\lambda, A} \colon C^\infty_R(\mathbb{S}^3) \longrightarrow C^\infty_R(\mathbb{S}^3), 
\qquad
f \mapsto S^*_{\lambda, A} S_\lambda f 
= (\rho^2 - \Delta) f,
\end{equation}
where $\Delta$ denotes the Laplacian on the base sphere $\mathbb{S}^2$, 
lifted to $\mathbb{S}^3$ via the Hopf fibration. 
Moreover, $\Delta_{\lambda, A}$ reduces to an invertible elliptic operator on $\SS^2$.\\
Following~\cite{EP13}, the \emph{contact bracket} on $\mathbb{S}^3$ 
for functions $f, g \in C^\infty_R(\mathbb{S}^3)$ is defined by
\begin{equation}
\{f, g\} := S_\lambda f(g) 
= \d\lambda(S_{\lambda}f, S_{\lambda}g),
\end{equation}
where $S_\lambda f$ is the contact vector field associated with $f$, 
as defined in~\eqref{eq: operator S}. \\
For a fixed smooth function $\varphi \colon \mathbb{S}^3 \to \mathbb{R}$, 
this bracket gives rise to a (trivial) Lie algebra $2$-cocycle, 
again as described in~\cite{EP13}:
\begin{equation}
\label{eq:trivial_cocycle}
\Omega(u, v) 
= \int_{\mathbb{S}^3} \varphi \, \{f, g\} \, \mathrm{d}\mathrm{vol}
= \int_{\mathbb{S}^3} \varphi \, \d\lambda(S_{\lambda}f, S_{\lambda}g) \, \mathrm{d}\mathrm{vol},
\end{equation}
where $u = S_\lambda f$ and $v = S_\lambda g$ 
are elements of the Lie algebra $\mathfrak{g}=T_{\id}\mathcal{D}_q\left(\Sd\right)$.
\subsection{The magnetic Euler--Arnold equation of the quantomorphism group}

We are now in a position to derive the magnetic Euler--Arnold equation 
associated with the geometric structure introduced in the previous subsection.

\begin{prop}\label{prop: magnetic Euler Arnold on quant}
Let \( a \in \mathbb{R} \) be fixed. 
The magnetic Euler--Arnold equation corresponding to the magnetic system 
\( \left( \mathcal{D}_q\left(\Sd\right), \mathcal{G}^{A}, a \cdot \Omega \right) \) 
is given by
\begin{equation}\label{eq: magnetic quant Euler Arnold}
    \partial_t \Delta_{\lambda, A} f 
    + \{f, \Delta_{\lambda, A} f\} 
    - a \{ \varphi, f \} = 0,
\end{equation}
where \( f \in C^\infty_R(\mathbb{S}^3) \).
\end{prop}

\begin{rem}
Alternatively, this equation can be derived using 
\Cref{p: geodesics on central extensions} 
and \Cref{C: relation magnetic geodesics and geodesics on central extension}, 
in combination with~\cite[Eq.~19]{ModinSurin2025}. 
It is important to emphasize, however, that the validity of 
\cite[Eq.~19]{ModinSurin2025} depends on the existence of a one-dimensional 
central extension of the quantomorphism group. 

In contrast, the existence of the Lorentz force~\eqref{eq: lorenz quant group} 
as well as the adjoint of \( \mathrm{ad} \) in~\eqref{eq: adjoint quantomorph group} 
does not rely on the presence of such a central extension.
\end{rem}

\begin{proof}
By~\cite[Lemma~3.3]{ModinSurin2025}, 
the Lorentz force evaluated at the identity element of the magnetic system 
\( \left( \mathcal{D}_q\left(\Sd\right), \mathcal{G}^A, \Omega \right) \) 
is given by
\begin{equation} \label{eq: lorenz quant group}
Y_{\mathrm{id}} \colon \mathfrak{g} \longrightarrow \mathfrak{g}, 
\qquad 
u \mapsto S_{\lambda} \left( \Delta_{\lambda, A}^{-1} \{ \varphi, f \} \right),
\end{equation}
where \( u = S_{\lambda} f \). Furthermore, by~\cite[Eq.~18]{ModinSurin2025}, 
the adjoint of the Lie algebra operator satisfies
\begin{equation} \label{eq: adjoint quantomorph group}
\mathrm{ad}_u^T(u) 
= S_{\lambda} \left( 
\Delta_{\lambda, A}^{-1} \{ f, \Delta_{\lambda, A} f \} 
\right), 
\qquad \forall u = S_{\lambda} f \in \mathfrak{g}.
\end{equation}
Using the general form of the magnetic Euler--Arnold equation 
from \Cref{t: magnetic Euler equation on g}, 
and substituting~\eqref{eq: lorenz quant group} 
and~\eqref{eq: adjoint quantomorph group}, we obtain
\begin{align*}
0 
&= \partial_t S_{\lambda} f 
+ S_{\lambda} \left( 
\Delta_{\lambda, A}^{-1} \{ f, \Delta_{\lambda, A} f \} 
\right) 
- a \, S_{\lambda} \left( 
\Delta_{\lambda, A}^{-1} \{ \varphi, f \} 
\right) \\
&= S_{\lambda} \Delta_{\lambda, A}^{-1} 
\left( 
\partial_t \Delta_{\lambda, A} f 
+ \{ f, \Delta_{\lambda, A} f \} 
- a \{ \varphi, f \} 
\right).
\end{align*}
Since \( \Delta_{\lambda, A} \) is invertible and 
\( S_{\lambda} \) is injective on \( C^\infty_R(\mathbb{S}^3) \), 
the above equation is equivalent to 
\eqref{eq: magnetic quant Euler Arnold}, 
which completes the proof.
\end{proof}
\subsection{The global quasi-geostrophic equations as magnetic Euler--Arnold equation}

Following the line of reasoning in~\cite[Rmk.~3.4]{ModinSurin2025}, 
we derive from \Cref{prop: magnetic Euler Arnold on quant} 
the global quasi-geostrophic equations as a magnetic Euler--Arnold equation. 
We begin by choosing an $\mathbb{S}^1$-invariant function 
\( \varphi \) in~\eqref{eq:trivial_cocycle}, 
that is, \( \varphi \in C^{\infty}_R(\mathbb{S}^3) \).

More precisely, for a given smooth function 
\( h \colon \mathbb{S}^2 \to \mathbb{R} \), 
we consider the function
\begin{equation}
\label{eq:phi-map}
\varphi \colon \mathbb{S}^2 \to \mathbb{R}, 
\quad (z, w) \mapsto \frac{2z}{\mathrm{Ro}} + \frac{2z h(z, w)}{\mathrm{Ro}},
\end{equation}
where \( \rho^2 = \gamma z^2 \), and 
\( \gamma \) and \( \mathrm{Ro} \) are as described in~\eqref{eq: global quasi-geostrophic equations}. 
We denote by \( \varphi \) its lift to \( \mathbb{S}^3 \) via the Hopf fibration.

For \( a = 1 \), and using~\eqref{eq: contact Laplacian}, 
equation~\eqref{eq: magnetic quant Euler Arnold} 
reduces to the equation on \( \mathbb{S}^2 \)
\begin{equation}
\label{eq:reduced-eq}
\partial_t\big((\gamma z^2 - \Delta) f\big) 
+ \left\{ f,\, (\gamma z^2 - \Delta) f 
+ \frac{2z}{\mathrm{Ro}} + 2z h \right\} = 0.
\end{equation}

Defining \( \psi = f \) and
\[
q = (\gamma z^2 - \Delta) f 
+ \frac{2z}{\mathrm{Ro}} + 2z h,
\]
and comparing with~\eqref{eq: global quasi-geostrophic equations}, 
we obtain from~\eqref{eq:reduced-eq} the following.

\begin{cor}\label{c: global quasi-geostrophic equations as magnetic}
The global quasi-geostrophic equations~\eqref{eq: global quasi-geostrophic equations} 
are the magnetic Euler--Arnold equations of the magnetic system 
\( \left( \mathcal{D}_q, \mathcal{G}^{A}, \Omega \right) \).\\
In particular, a curve \( \gc \) is a magnetic geodesic in 
\( \left( \mathcal{D}_q, \mathcal{G}^{A}, \Omega \right) \) 
if and only if the Eulerian velocity field 
\( u := \dgc \circ \gc^{-1} \) 
is a solution of~\eqref{eq: global quasi-geostrophic equations}.
\end{cor}

\begin{rem}\label{r: lorenz force global quasi}
By comparing \eqref{eq: lorenz quant group} and \eqref{eq:reduced-eq}, 
we may interpret the correction term 
\( \frac{2z}{\mathrm{Ro}} + 2z h \) in 
\eqref{eq: global quasi-geostrophic equations} 
as the Lorentz force acting on a charged particle 
in the infinite-dimensional magnetic system 
\( \left( \mathcal{D}_q, \mathcal{G}^{A}, \Omega \right) \).
\end{rem}

\subsection{Local and global well-posedness of \eqref{eq: global quasi-geostrophic equations}}

In the classical work of Ebin and Marsden~\cite{EM70}, 
building on Arnold’s geometric interpretation of the Euler equations~\cite{Arnold66}, 
well-posedness results were obtained via the geodesic formulation, 
establishing both local and global existence.

In this spirit, we adopt a similar geometric framework to establish 
well-posedness for the magnetic geodesic flow on 
\( \left( \mathcal{D}^{s}_q(\Sd), \mathcal{G}^{A}, \Omega \right) \). 
We fix the following notation: 
\( H^{s}_{R}(\mathbb{S}^d) \) denotes the Sobolev space of order \( s \) 
consisting of functions invariant under the Reeb flow, 
and \( \mathcal{D}^{s}_q(\Sd) \) denotes the group of quantomorphisms 
of Sobolev class \( s \).

The local well-posedness of the magnetic geodesic flow on 
\( \left( \mathcal{D}^{s}_q(\Sd), \mathcal{G}^{A}, a \cdot \Omega \right) \) 
follows line by line from \cite[Thm.~3.1]{ModinSurin2025}, 
which in turn builds on \cite{EP13,EP15}.

\begin{cor}
Magnetic geodesics of 
\( \left( \mathcal{D}^{s}_q(\Sd), \mathcal{G}^{A}, a\cdot \Omega \right) \) 
exist locally in time in the sense of the Picard--Lindelöf theorem. 
More precisely, for any initial data, there exists a non-empty maximal time interval 
\( (-T_a, T_b) \) on which a unique solution exists and depends smoothly 
on the initial data.
\end{cor}

Furthermore, one obtains the following global well-posedness result by an argument 
that follows line by line from \cite[Thm.~3.2, Rmk.~3.4]{ModinSurin2025}:

\begin{cor}
For initial data \( f_0 \in H^{s+1}(\mathbb{S}^2, \mathbb{R}) 
\simeq H^{s+1}_{R}(\Sd) \) with \( s > 2 \), 
the solution of \eqref{eq: magnetic quant Euler Arnold} 
exists for all time.
\end{cor}

We conclude this section with a speculative remark.

\begin{rem}\label{rem: speculative}
It is well known that exact magnetic systems—i.e., magnetic systems 
\( (M, g, \sigma) \) in which the magnetic field is an exact two-form 
\( \sigma = \mathrm{d}\alpha \)—admit a Lagrangian formulation 
via an action functional (see, for example,~\cite{Abbo13Lect}). 
Critical points of this action functional correspond to magnetic geodesics.\\
In the case of the incompressible Euler equations, 
the action-functional formulation has led to the development of measure-valued solutions; 
see~\cite{DaneriFigalli2013} and the references therein. 
It would therefore be interesting to investigate whether a similar variational approach 
could be employed to define measure-valued solutions of 
\eqref{eq: global quasi-geostrophic equations}.
\end{rem}
\noindent\\
\textbf{Data availability} In this research no data is processed.\\
\noindent\\ 
\textbf{Declarations}\\\\
\textbf{Conflict of interest:} The author states that there is no conflict of interest.
\bibliographystyle{abbrv}
\bibliography{ref}
    \end{document}